\documentclass[oneside,final]{csri24}

\usepackage[top=1.355in,
			bottom=1.355in,
			left=1.5in,
			right=1.5in]{geometry}
\usepackage{amsfonts,
			amsmath,
			graphicx,
			subfigure,
			url}
\usepackage{algorithm,algpseudocode}
\usepackage{amssymb}
\usepackage{mathtools}
\usepackage{booktabs}

\usepackage{framed}
\usepackage{graphicx}
\usepackage{tikz}
\usepackage[colorlinks=true, allcolors=blue]{hyperref}
\newcommand{\Prox}{\mathrm{Prox}}

\newcommand{\Proxrphi}{\mathrm{Prox}_{r\phi}}
\newcommand{\Proxrzerophi}{\mathrm{Prox}_{r_{0}\phi}}

\newcommand{\Argmin}{\operatorname*{arg\,min}}
\newcommand{\R}{\mathbb{R}}
\newcommand{\Rn}{\mathbb{R}^n}
\newcommand{\dx}{\textrm{d}x}

\newtheorem{assumption}{Assumption}[section]
\newcommand{\kgrad}{\kappa_{\mathrm{grad}}}
\newcommand{\kobj}{\kappa_{\mathrm{obj}}}
\newcommand{\kfcd}{\kappa_{\mathrm{fcd}}}
\newcommand{\krad}{\kappa_{\mathrm{rad}}}

\newcommand{\ared}{\mathrm{ared}}
\newcommand{\pred}{\mathrm{pred}}
\newcommand{\cred}{\mathrm{cred}}
\newcommand{\dom}{\mathrm{dom }}

\DeclarePairedDelimiterX{\inp}[2]{\langle}{\rangle}{#1, #2}
\DeclarePairedDelimiterX{\norm}[1]{\|}{\|}{#1}

\title{A{N} {Inexact} Weighted Proximal Trust-Region Method$^1$}

\author{Leandro Farias Maia\thanks{Department of Industrial and Systems Engineering, Texas A\&M University,  leandro.maia@tamu.edu} \and Robert Baraldi\thanks{Optimization and Uncertainty Quantification, Sandia National
Laboratories, rjbaral@sandia.gov } \and Drew P. Kouri\thanks{Optimization and Uncertainty Quantification, Sandia National
Laboratories, dpkouri@sandia.gov}}

\begin{document}

\maketitle
\footnotetext[1]{This work was previously published in the Summer Proceedings of the Center for Computing Research at Sandia National Laboratories. The present version is posted on arXiv for broader dissemination.}

\begin{abstract}
In [R.~J.~Baraldi and D.~P.~Kouri, {\em Math.~Program.}, 201:1 (2023), pp.~559-598], the authors introduced a trust-region method for minimizing the sum of a smooth nonconvex and a nonsmooth convex function, the latter of which has {an} analytical proximity operator. 
While many functions satisfy this criterion, e.g., the $\ell^1$-norm defined on $\ell^2$, many others are precluded by either the topology or the nature of the nonsmooth term.
Using the $\delta$-Fr\'{e}chet subdifferential, we extend the definition of the inexact proximity operator and enable its use within the aforementioned trust-region algorithm. Moreover, we augment 
the analysis for the standard trust-region convergence theory to handle proximity operator inexactness 
with weighted inner products.  We first introduce an algorithm to generate a point in the inexact proximity operator and then apply the algorithm within the trust-region method to solve an optimal control problem constrained by Burgers' equation.

\end{abstract}

\section{Introduction}
Our goal is the efficient numerical solution to the nonsmooth, nonconvex optimization problem
\begin{equation}\label{eq:optprob}
  \min_{x\in X}\; \left\{f(x)+\phi(x)\right\},
\end{equation}
where $X$ is a real Hilbert space, $f:X\to\R$ is a smooth nonconvex function and $\phi:X\to(-\infty,+\infty]$ is a proper, closed and convex function. 
The method developed in \cite{Baraldi2023} solves \eqref{eq:optprob} via a trust-region method that permits inexact evaluations of the smooth objective function $f$ and its gradient.
The convergence analysis for the trust-region method is contingent on the proximity operator of $\phi$ being exact, i.e. analytically
computable. 
In this paper, we extend the inexact proximal trust-region algorithm in \cite{Baraldi2023} to leverage certain inexact evaluations of the proximity operator of $\phi$.
The proximity operator of $\phi$ at $x$ with stepsize $r$ is defined as the unique solution to the minimization problem
\begin{equation}\label{eq:exact.prox}
  \Proxrphi(x) \coloneqq
  \Argmin_{y\in X}\left\{ \frac{1}{2r}\|y-x\|^2 + \phi(y)\right\}
\end{equation}
and is a powerful tool to deal with composite nonsmooth optimization problems like \eqref{eq:optprob}.
Algorithms that leverage the proximity operator, like ISTA and FISTA \cite{Beck17}, are typically first-order and only apply to convex optimization problems.
Like \cite{Baraldi2023}, these methods require that the proximity operator be computed exactly---a requirement that is often computationally expensive or even impossible to satisfy. For instance, if $X$ is finite dimensional with dot product defined by a diagonal matrix $D$ and $\phi$ is the $\ell^1$ norm, then the proximity operator has an analytical expression. 
In contrast, if the dot product is defined by a non-diagonal symmetric positive definite (spd) matrix $M$, then evaluating the proximity operator requires a specialized iterative method to obtain an approximate solution.
Although preliminary studies on optimization with inexact proximity operators {exist} \cite{Devolder14,MaiaInexactCyclic,Hua12,Kim-chuan12,salzo2012inexact,Richtarik14},
these works only apply to convex problems and often exhibit slow convergence rates.  Using these methods as motivation, we develop a proximal Newton method---tailored to the case of diagonal and non-diagonal dot products---that can handle inexact proximity operators.  We define our inexact proximity operator as
\begin{align}\label{intro:inexact.prox}
    \Proxrphi(x,\delta) \coloneqq  \Big\{p\in X\,\Big\vert\, &\frac{1}{2r}\|p-x\|^2+\phi(p) \nonumber \\
    &\le\frac{1}{2r}\|z-x\|^2+\phi(z)+\delta\|z-p\|\;\;\forall\,z\in X\Big\}
\end{align}
The inclusion of the term $\delta\|z-u\|\ge 0$ on the right-hand side of \eqref{intro:inexact.prox} allows the proximity operator to be computed inexactly.

\if(0)
Our goal in this paper is twofold:
\begin{itemize}
    \item[1.] Develop an algorithm to obtain one solution to the inexact proximal defined in \ref{intro:inexact.prox}, for a some classes of matrices $M$.
    \item[2.] Apply our novel inexact proximal map to \cite{Baraldi2023} to solve \textit{{Burgers'}} equation \textcolor{red}{[reference]}.
\end{itemize}
\fi






The remainder of the article is structured as follows.
In Section~\ref{sec:trust}, we review the inexact trust-region method introduced in \cite{Baraldi2023}, presenting the main components of the algorithm. 
We also include a lemma that enables the use of our inexact proximal map. In Section~\ref{sec:inexact.subgradient.proximal}, we lay out our definition of inexact proximity operator based on the Fr\'{e}chet subdifferential. 
In Section~\ref{sec:inexact.subgradient.proximal}, we present the main technical results of this work, where we establish the relationship between the $\delta$-Fr\'{e}chet subdifferential \cite{mordukhovich2006variational} and the $\delta$-proximal map. We conclude with Section~\ref{sec:M.weighted.prox}, where we present numerical results for the optimal control of Burgers' equation. 


\section{Notation and Problem Assumptions}
\label{sec:notation}
Throughout, $X$ denotes a real Hilbert space with inner product $\inp{x}{y}$ and norm $\|x\|=\inp{x}{x}^{1/2}$.  
To simplify the presentation, we associate the topological dual space $X^*$ with $X$ through Riesz representation.
We assume that $\phi:X\to(-\infty,+\infty]$ is proper, closed and convex with effective domain
\[
  \textup{dom}\,\phi\coloneqq \{x\in X\,\vert\, \phi(x)<+\infty\}.
\]
We further assume that $f:X\to\R$ is Fr\'{e}chet differentiable on an open set containing $\textup{dom}\,\phi$ and its derivative is Lipschitz continuous on that set.  We denote the derivative of $f$ at $x\in X$ by $f'(x)\in X^*$ and its gradient (i.e., Riesz representer) by $\nabla f(x)\in X$.  Finally, we assume that the total objective function $F(x)\coloneqq f(x)+\phi(x)$ is bounded from below.

The main focus of this paper is the use of inexact proximity operators arising from weighted inner products.  To this end, we define $a:X\times X\to\mathbb{R}$ to be a symmetric, coercive and continuous bilinear form, i.e., there exists $0<\alpha_1\le\alpha_2<\infty$ such that
\[
    a(x,y) = a(y,x),\quad \alpha_1\|x\|^2\le a(x,x), \qquad\text{and}\quad a(x,y) \le \alpha_2\|x\|\|y\| 
\qquad\forall\,x,\,y\in X.
\]
We associate $a$ with the invertible (cf.~the Lax-Milgram lemma), self-adjoint, positive and continuous linear operator $A:X\to X$ defined by its action:
\[
  \inp{Ax}{y} = a(x,y)\qquad\forall\,x,\,y\in X.
\]
Recall that $a(\cdot,\cdot)$ defines an inner product on $X$ and its associated norm $\|\cdot\|_a = \sqrt{a(\cdot,\cdot)}$ is equivalent to $\|\cdot\|$.  We denote the (inexact) proximity operator associate with $a$ by $\Proxrphi^a$.

For our numerical results, we will work with $X=\R^n$ endowed with the dot product
\[
  \inp{x}{y}=\inp{x}{y}_M \coloneqq  x^\top M y = \sum_{i=1}^n\sum_{j=1}^n m_{i,j}x_i y_j,
\]
where $M\in\R^{n\times n}$ is a non-diagonal spd matrix.  In many common choices of $\phi$, the proximity operator associated with the $M$-weighted dot product lacks an analytical form.  In this case, we replace $\inp{\cdot}{\cdot}$ with the equivalent dot product
\[
  a(x,y) = \inp{x}{y}_D \coloneqq  x^\top D y = \sum_{i=1}^n d_i x_i y_i,
\]
where $D=\textup{diag}(d)\in\R^{n\times n}$ is a positive diagonal matrix.  In this setting, the linear operator $A$ is given by
\[
  A = M^{-1}D.
\]
For clarity, we denote $\Rn$ endowed with the $M$-weighted dot product by $X_M$ and analogously for $X_D$.
We further denote the proximity operator defined on $X_M$ by $\Proxrphi^M$ and similarly on $X_D$ by $\Proxrphi^D$.



\section{A Proximal Trust{-}Region Method}
\label{sec:trust}

Trust{-}regions are iterative methods for computing approximate solutions to general nonconvex optimization problems \cite{conn2000trust} and are ideal for leveraging inexact computations \cite{Baraldi2023,kouri2018inexact}.
In the recent work \cite{Baraldi2023}, the authors developed a proximal trust-region method for nonsmooth optimization problems with the form \eqref{eq:optprob} that exploits inexact function and gradient evaluations with guaranteed convergence.  At the $k$-th iteration of \cite[Algorithm~1]{Baraldi2023}, we compute a trial iterate $x_k^+$ that approximately solves the trust-region subproblem
\begin{equation}\label{eq:tr-sub}
  \min_{x\in X}\{m_k(x)\coloneqq f_k(x)+\phi(x)\} \qquad \text{subject to}\qquad \|x-x_k\|\le \Delta_k,
\end{equation}
where $x_k \in \textup{dom}\,\phi$ is the current iterate, $f_k$ is a smooth local model of $f$ around $x_k$, and $\Delta_k > 0$ is the trust-region radius. In particular, we require that the trial iterate $x_k^+$ satisfies the fraction of Cauchy decrease condition: there exist positive constants $\krad$ and $\kfcd$, independent of $k$, such that
\begin{subequations}\label{eq:FCD.condition}
\begin{align}
  \|x_k^+-x_k\| &\le \krad \Delta_k \\
  m_k(x_k)-m_{k}(x_k^+) &\ge \kfcd h_k\min\left\{\frac{h_k}{1+\omega_k},\Delta_k\right\},
\end{align}
\end{subequations}
where $\omega_k$ is a measure of the curvature of $f_k$ given by
\begin{equation*}\label{eq:def.curvature}
  \omega_k\coloneqq \sup \{|\omega(f_k,x_k,s)| \,\vert\, 0<\|s\|\le \krad\Delta_k \},
\end{equation*}
$\omega(g,x,s)$ is the curvature of the Fr\'{e}chet differentiable function $g:X\to\R$ at $x\in X$ in the direction $s\in X$, i.e.,
\[
  \omega(g,x,s) \coloneqq  \frac{2}{\|s\|^2}[g(x+s)-g(x)-\inp{\nabla g(x)}{s}],
\]
and $h_k$ is our stationarity metric given by
\begin{equation}
    \label{eq:hk_exact}
  h_k \coloneqq \frac{1}{r_0}\|\Proxrzerophi(x_k-r_0\nabla f(x_k))-x_k\|.
\end{equation}
    
Given a trial iterate $x_k^+$ that satisfies \eqref{eq:FCD.condition}, we decide whether to accept or reject $x_k^+$ using the ratio of computed ($\cred_k$) and predicted ($\pred_k$) reductions,
\begin{align}\label{eq:rho}
  \rho_k\coloneqq\frac{\cred_k}{\pred_k}.
\end{align}
The computed reduction $\cred_k$ is an approximation of the actual reduction
\begin{align}
  \ared_k \coloneqq F(x_{k})-F(x_k^+) 
\end{align}
and the predicted reduction $\pred_k$ is the decrease predicted by the model $m_k$,
\begin{align}
  \pred_k \coloneqq m_k(x_k)-m_k(x_k^+).
\end{align}
If $\rho_k\ge \eta_1$, we accept $x_k^+$, i.e., $x_{k+1}=x_k^+$. Otherwise, we reject it, setting $x_{k+1}=x_k$. We furthermore use $\rho_k$ to increase or decrease the trust-region radius $\Delta_k$.  Here, $\eta_1\in(0,1)$ is a user-specified parameter.

Algorithm~\ref{alg:cap} states the inexact proximal trust-region method from \cite[Algorithm~1]{Baraldi2023}.
To ensure convergence, Algorithm~\ref{alg:cap} requires the following assumptions on the inexact evaluations $f$ and its gradient. The first assumption ensures that the computed reduction $\cred_k$ is a sufficiently accurate approximation of the actual reduction $\ared_k$.

\begin{assumption}[Inexact Objective]
\label{assump:inexact.object}
    There exists a positive constant $\kobj$, independent of $k$, such that
    \begin{equation}\label{eq:assump.inexact.objective}
        |\ared_k-\cred_k| \leq \kobj\left[ \eta\min\{\pred_k,\theta_k\}\right]^{\zeta}
    \end{equation}
where
\[
\zeta>1,\quad 0<\eta<\min\{\eta_1,(1-\eta_2)\}\quad \text{and}\quad \lim_{k\to+\infty}\theta_k=0
\]
are used provided parameters.
\end{assumption}

{Note that all quantities on the right-hand side of~\eqref{eq:assump.inexact.objective} are available when
computing $\cred_k$, enabling one to avoid the computation of the actual reduction $\ared_k$ and hence the objective function $F$.}
Similarly, the following assumption enables approximations of the gradient of $f$ within a prescribed tolerance that depends on the state of the algorithm.

\begin{assumption}[Inexact Gradient]
\label{assump:inexact.gradient}
The model $f_k: X \rightarrow \R$ has Lipschitz continuous derivatives on an open set containing $\dom\, \phi$ and there exists a positive constant $\kgrad$, independent of $k$, such that
\[
\|g_k - \nabla f(x_k)\|\leq \kgrad\min \{h_k,\Delta_k\},
\]
where $g_k \coloneqq \nabla f_k(x_k)$.
\end{assumption}

\begin{algorithm}
\caption{Nonsmooth Trust-Region Algorithm}\label{alg:cap}
\begin{algorithmic}[1]
\Require $x_1\in \dom \psi$, $\Delta_1>0$, $0<\eta_1<\eta_2<1$, and $0<\gamma_1\leq \gamma_2 \leq \gamma_3$
\For{$k=1,2,\ldots$}
    \State \textbf{Model Selection:} Choose $m_k$ that satisfies Assumption~\ref{assump:inexact.gradient} \label{alg:assump.model.inexact}
    \State \textbf{Step Computation:} Compute $x_k^+\in X$ satisfying equation~\eqref{eq:FCD.condition} \label{alg:exact.inexact}
    \State \textbf{Computed Reduction:} Compute $\cred_k$ satisfying Assumption~\ref{assump:inexact.object}
    \State \textbf{Step Acceptance and Radius Update:} Compute $\rho_k$ as in \eqref{eq:rho}

    \If{$\rho_k<\eta_1$}
        \State $x_{k+1}\gets x_k$
        \State $\Delta_{k+1}\in [\gamma_1\Delta_k,\gamma_2\Delta_k]$

    \Else
        \State $x_{k+1}\gets x_k^+$
        \If{$\rho_k\in [\eta_1,\eta_2)$}
            \State $\Delta_{k+1}\in[\gamma_2\Delta_k,\Delta_k]$
        \Else
            \State $\Delta_{k+1}\in[\Delta_k,\gamma_3\Delta_k]$
        \EndIf
    \EndIf

\EndFor
\end{algorithmic}
\end{algorithm}






A critical component of the convergence theory for Algorithm~\ref{alg:cap} is that the trial iterate satisfies \eqref{eq:FCD.condition}.
In the smooth case (i.e., $\phi\equiv 0$), this condition is satisfied if $x_k^+$ produces at least a fraction of the decrease achieved by the steepest descent step.
In the nonsmooth case, \cite{Baraldi2023} instead defines sufficient decrease using the proximal gradient step.  
In particular, \cite[Lemma~7]{Baraldi2023} illustrates the relationship between the proximal gradient stepsize, $r$, and the required model decrease. 
Before stating this important lemma, we first define the proximal gradient path
\begin{equation}\label{eq:path}
p_k(r)\coloneqq \Proxrphi(x_k-rg_k)-x_k \quad \text{and} \quad x_{k}(r)\coloneqq x_{k}+p_k(r).
\end{equation}
We also define the quantity
\begin{equation}\label{eq:model.decrease}
    Q_k(r)\coloneqq\inp{g_k}{p_k(r)}+\left(\phi(x_k(r)) - \phi(x_k)\right),
\end{equation}
which is used to measure the model decrease of the iterate $x_k(r)$.

The main descent lemma derived in \cite{Baraldi2023} plays the same role in our analysis, and we defer the proof of this lemma to maintain the focus of our exposition. We directly quote \cite[Lemma~7]{Baraldi2023} below.
\begin{lemma}
\label{lem:qk}
    The function $r \mapsto Q_k (r )$ is continuous and nonincreasing. In fact, if $r >t > 0$ then $Q_k (t) > Q_k (r )$ whenever $x_k (t) \neq x_k(r)$. Moreover, if $h_k > 0$ then
\[
Q_{k}(r)=\inp{g_k}{p_k(r)} + \left(\phi(x_k(r))-\phi(x_k)\right) \leq -\frac{1}{r}\|p_k(r)\|^2 < 0
\]
\end{lemma}

Lemma~\ref{lem:qk} is an inherent property of the proximal gradient map, which is used in \cite{Baraldi2023} to verify the conditions defining the Cauchy point
\cite[Eq.~(25)]{Baraldi2023}.
In the subsequent sections, we extend Algorithm~\ref{alg:cap} and Lemma~\ref{lem:qk} to handle the case of inexact proximity operators.

\section{Inexact Proximity Operator}
\label{sec:inexact.subgradient.proximal}
We first present a general definition of inexact proximity operators, which is motivated by the $\delta$-Fr\'{e}chet subgradient.  We will later demonstrate how to compute a $\delta$-proximity operator for the case of weighted inner products.

\begin{definition}[$\delta$-Proximity Operator]
\label{def:inexactProximal}
Let $x\in\text{dom}\,\phi$, $r>0$, and $\delta>0$ be given. We define the $\delta$-proximity operator at $x$ as the following set-valued map
\begin{align}\label{eq:inex.prox}
    \Proxrphi(x,\delta) \coloneqq  \Big\{p\in X\,\Big\vert\, &\frac{1}{2r}\|p-x\|^2+\phi(p) \nonumber \\
    &\le\frac{1}{2r}\|z-x\|^2+\phi(z)+\delta\|z-p\|\;\;\forall\,z\in X\Big\}.
\end{align}
\end{definition}

As claimed, the notion of $\delta$-proximity operator introduced in Definition~\ref{def:inexactProximal} is closely related the the $\delta$-Fr\'{e}chet subdifferential, which we define next.

\begin{definition}[$\delta$-Subgradient]
\label{def:subgradient}
    Let $x\in\text{dom}\,\phi$ and $\delta>0$ be given. We say that $s\in X$ is a $\delta$-Fr\'{e}chet subgradient, or simply a $\delta$-subgradient, of $\phi$ at $x$ if
    \[
    \phi(y)\ge \phi(x) + \inp{s}{y-x}-\delta\|y-x\|
    \]
for all $y\in X$. The set $\partial \phi_{\delta}(x)$ consists of all $\delta$-subgradient of $\phi$ at $x$.
\end{definition}

It is straightforward to observe that by setting $\delta=0$ we recover the standard definitions of subgradient and proximity operator.
By replacing the proximity operator with the inexact version defined in Definition~\ref{def:inexactProximal}, we redefine the proximal gradient path \eqref{eq:path}, and hence the Cauchy point, as
\begin{equation}
\tilde p_k(u,r)\coloneqq u-x_k \quad \text{and} \quad \tilde x_{k}(u,r)\coloneqq x_{k}+\tilde p_k(u,r),
\end{equation}
where $u\in \Proxrphi (x_k-rg_k,\delta)$. Analogously, we update the quantity $Q_k$ in \eqref{eq:model.decrease} to
\begin{equation}
    \tilde Q_k(u,r)\coloneqq\inp{g_k}{\tilde p_k(u,r)}+\left(\phi(\tilde x_k(u,r)) - \phi(x_k)\right)
\end{equation}

We derive the close relationship between the $\delta$-subdifferntial and the $\delta$-proximity operator in Theorem~\ref{thm:delta.Prox.Thm}, but we first require the following technical lemma.
\if(0)
\begin{lemma}[Optimality Condition]
\label{lemma:optim.subg} 
The following holds for $x\in X$ and $\delta\ge 0$,
\[
  \phi(x)-\phi(y)\leq \delta \|x-y\| \quad\forall\, y\in X
  \qquad\iff\qquad
  0\in\partial_\delta \phi(x).
\]
\end{lemma}

\begin{proof}
This is a direct consequence of Definition~\ref{def:subgradient}.
\end{proof}

\begin{lemma}\label{lem:inexact.subg.prox}
Let $x,\, z \in X$ and $\delta\ge 0$. Then,
\[
  u\in z + \partial_\delta \phi(x)
  \qquad\iff\qquad
  u-z \in \partial_\delta \phi(x).
\]
\end{lemma}

\begin{proof}
We begin by observing that $u \in \partial_\delta \phi(x)+z$ if, and only if, $u=\tilde u+z$, where $\tilde u\in \partial_\delta\phi(x)$. Using Definition~\ref{def:subgradient}, the last condition is equivalent to
\begin{align*}
\phi(y)&\geq \phi(x)+\inp{\tilde u}{y-x}-\delta\|x-y\|\\
&= \phi(x)+\inp{u-z}{y-x}-\delta\|x-y\|
\end{align*}
for every $y\in X$, and this concludes the proof.
\end{proof}
\textcolor{red}{Do we need the proof of lemma 5.2? This is a result of Clason/Valkonen book? In fact, why not use the C\&V book results? 5.1 is the definition and 5.2 is trivial.}
\fi
\begin{lemma}\label{lem:composite.inexact}
Let $g:X\to\R$ be a continuously Fr\'{e}chet differentiable convex function,
\[
  H(x)\coloneqq g(x) + \phi(x),
\]
and $\delta\ge 0$, then
\[
  \partial_\delta H(x) = \nabla g(x) + \partial_\delta \phi(x).
\]
Moreover, $x^*$ is a $\delta$-minimizer of $H$, i.e., $H(x^*)\le H(x) + \delta\|x-x^*\|$ for all $x\in X$, if and only if 
\[
  -\nabla g(x^*) \in \partial_\delta \phi(x^*).
\]
\end{lemma}
\begin{proof}
This is a direct consequence of \cite[Corollary 17.3]{clason2023introduction}.
\end{proof}

We now state and prove the main result of this section.

\begin{theorem}
\label{thm:delta.Prox.Thm}
Let $x\in X$ and $\delta\ge 0$. Then,
\[
  u\in \Proxrphi(x,\delta )
  \qquad\iff\qquad
  x \in (I+r\partial_\delta\phi)(u).
\]
\end{theorem}
\begin{proof}
Let
\(
\varphi(\cdot ) = \tfrac{1}{2r}\|\cdot -x\|^2  + \phi(\cdot ),
\)
and recall Definition~\ref{def:inexactProximal} implies $u\in \Proxrphi(x,\delta)$ satisfies 
\begin{equation}\label{eq:delta.prox.H}
    \varphi(u) \le \varphi(z) + \delta\|z-u\| \quad\forall\, z\in X
\end{equation}
i.e., $u$ is a $\delta$-minimizer of $\varphi$.
Applying Lemma~\ref{lem:composite.inexact}, expression \eqref{eq:delta.prox.H} implies
\[
-\nabla g(u)= r^{-1}(x-u)\in \partial_\delta\phi(u)
\quad \iff \quad 
x \in u+r\partial_\delta\phi(u) = (I+r\partial_\delta\phi)(u),
\]
concluding the proof.
\end{proof}

 
The proximity operator is present in two components of Algorithm~\ref{alg:cap}, through the stationarity metric $h_k$ and in defining the Cauchy point using $Q_k$.
To ensure convergence, we require that $\delta>0$ and $u\in \Proxrphi(x_k-rg_k,\delta)$ are chosen so that the Cauchy point produces sufficient decrease.  
When using inexact proximity operators, the stationarity metric \eqref{eq:hk_exact}
becomes 
\begin{equation}
\label{eq:hk_inexact}
  \tilde h_k \coloneqq \frac{1}{r_0}\|\tilde p_k(u,r_0)\|
  \qquad\text{for some}\qquad
  u\in\Proxrphi(x_k-r_0 g_k,\delta).
\end{equation}
To address the accuracy of $\tilde h_k$, we notice that the reverse triangle inequality ensures that
\[
  h_k \le \tilde h_k + |h_k-\tilde h_k|.
\] 
Consequently, by choosing $\delta>0$ so that
\begin{align}\label{eq:h.inexactness}
|h_k  -  \tilde h_k|\le \kgrad \min\{\tilde h_k,\Delta_k\},
\end{align}
holds, we ensure that
\[
  h_k \le (1+\kgrad)\tilde h_k.
\]
Hence, if the limit inferior of $\tilde h_k$ is zero, then so is the limit inferior of $h_k$.
Given the computable quantity $\tilde h_k$, we further modify the trust-region algorithm to ensure that the trial iterate $x_k^+$ satisfies
\begin{subequations}\label{eq:exact.inexact}
\begin{align}
  \| x_k^+-x_k\| &\le \krad \Delta_k \\
  m_k(x_k)-m_{k}( x_k^+) &\ge \kfcd  \tilde h_k\min\left\{\frac{\tilde  h_k}{1+\omega_k},\Delta_k\right\}.
\end{align}
\end{subequations}
Finally, we modify Assumption~\ref{assump:inexact.gradient} to account for $\tilde h_k$ as follows.

\begin{assumption}[Inexact Gradient]
\label{assump:inexact.gradient.prox}
The model $f_k: X \to \R$ has Lipschitz continuous gradient on an open set containing $\dom\, \phi$ and there exists a positive constant $\kgrad$, independent of $k$, such that
\[
\begin{aligned}
|h_k - \tilde h_k| &\le \kgrad \min\{\tilde h_k, \Delta_k\} \\
\|g_k - \nabla f(x_k)\| &\le \kgrad\min \{\tilde h_k,\Delta_k\}.
\end{aligned}
\]
\end{assumption}

With these changes we are able to produce a new version of Algorithm~\ref{alg:cap} that allows inexact computation of the proximity operator.
In particular, line~2 in Algorithm~\ref{alg:cap} is modified to {\em ``Choose $\tilde{h}_k$ and $m_k$ that satisfy Assumption~\ref{assump:inexact.gradient.prox}''}.
Notice that $\tilde h_k$ and $g_k$ must be computed in tandem to satisfy the conditions in Assumption~\ref{assump:inexact.gradient.prox}, which can be accomplished by a straight forward modification of \cite[Algorithm~4]{Baraldi2023}.
In order to ensure that the Cauchy point achieves sufficient decrease, we additionally enforce the proximal gradient descent condition
\begin{equation}\label{eq:decr.inexact}
  \tilde Q_k(u,r) \le -\frac{\kappa_{\rm dec}}{r}\|\tilde{p}_k(u,r)\|^2,
\end{equation}
where $\kappa_{\rm dec}$ is a positive constant that is independent of $k$.  Note that if $u=\Proxrphi(x_k-rg_k)$, then \eqref{eq:decr.inexact} is satisfied with $\kappa_{\rm dec}=1$ (cf.~Lemma~\ref{lem:qk}).  To demonstrate how we enforce these conditions, we restrict our attention to the case of weighted inner products.

\if(0)
\begin{algorithm}
\caption{Nonsmooth Trust-Region Algorithm: Inexact Proximal}\label{alg:cap2}
\begin{algorithmic}[1]
\Require $x_1\in \dom \psi$, $\Delta_1>0$, $0<\eta_1<\eta_2<1$, and $0<\gamma_1\leq \gamma_2 \leq \gamma_3$
\For{$k=1,2,\ldots$}
    \State \textbf{Inexact Prox:} Choose $\tilde h_k$ satisfying~\eqref{eq:h.inexactness} \label{alg:tilde.h.k}
    \State \textbf{Model Selection:} Choose $m_k$ that satisfies Assumption~\ref{assump:inexact.gradient.prox} \label{alg:assump.model.inexact}
    \State \textbf{Step Computation:} Compute $x_k^+\in X$ satisfying equation~\eqref{eq:exact.inexact} \label{alg:exact.inexact}
    \State \textbf{Computed Reduction:} Compute $\cred_k$ satisfying Assumption~\ref{assump:inexact.object}
    \State \textbf{Step Acceptance and Radius Update:} Compute $\rho_k$ as in \eqref{eq:rho}

    \If{$\rho_k<\eta_1$}
        \State $x_{k+1}\gets x_k$
        \State $\Delta_{k+1}\in [\gamma_1\Delta_k,\gamma_2\Delta_k]$

    \Else
        \State $x_{k+1}\gets x_k^+$
        \If{$\rho_k\in [\eta_1,\eta_2)$}
            \State $\Delta_{k+1}\in[\gamma_2\Delta_k,\Delta_k]$
        \Else
            \State $\Delta_{k+1}\in[\Delta_k,\gamma_3\Delta_k]$
        \EndIf
    \EndIf

\EndFor
\end{algorithmic}
\end{algorithm}
\fi




\section{Weighted Proximal Operators}
\label{sec:M.weighted.prox}
To motivate this section, consider the finite dimensional case $X=X_M$.  When $\phi$ is separable, it is typically much easier to compute the proximity operator defined on $X_D$ (recall $D$ is a diagonal matrix), i.e.,
\[
  \Prox^D_{r \phi}(x) = \operatorname*{arg\,min}_{y\in X} \left\{\frac{1}{2r}\|y - x\|_D^2 + \phi(y)\right\},
\]
because the optimization problem defining $\Prox^D_{r\phi}(x)$ can be reduced to solving one-dimensional optimization problems for each component of $\Prox^D_{r\phi}(x)$.
For example, this is the case when $\phi(\cdot) \equiv \|\cdot\|_1$ or when the $\phi$ is the indicator function for bound constraints. 

We consider the more general setting of replacing $\Proxrphi$ with $\Proxrphi^a$ defined using the inner product induced by the bilinear form $a(\cdot,\cdot)$ and assume that $\Proxrphi^a$ has an analytical form.
To this end, we define an algorithm, listed as Algorithm~\ref{alg:inex.M.prox} that computes a $\delta$-proximity operator using $\Proxrphi^a$.
Algorithm~\ref{alg:inex.M.prox} is the usual $a$-weighted proximal gradient algorithm applied to compute $\Proxrphi(x)$ and consequently, the iterates $\{u_\ell\}$, ignoring the stopping conditions, converge strongly to $\Proxrphi(x)$ \cite[Corollary~28.9]{bauschke2017monotone}.
Notice that we employ a modified stopping condition that helps to ensure that the computed solution is a $\delta$-proximity operator.

\begin{algorithm}[htb!]
\caption{Weighted Proximal Gradient}
\begin{algorithmic}[1]
\Require{The point $x\in X$, the proximity parameter $r>0$, and $\varepsilon>0$}
\State $\ell\gets 0$
\State $u_0\gets \Prox^a_{r\phi}(x)$
\While{$\varepsilon < \|u_\ell-u_{\ell+1}\|_a$}
  \State{$u_{\ell+1}\gets \Prox_{r\phi}^a(u_{\ell}-A^{-1}(u_{\ell}-x))$}
  \State{$\ell\gets \ell+1$} 
\EndWhile
\end{algorithmic}
\label{alg:inex.M.prox}
\end{algorithm}



To satisfy Assumption~\ref{assump:inexact.gradient.prox}, we must bound $|h_k-\tilde h_k|$, which we can bound in terms of the error in the $\delta$-proximity operator using the reverse triangle inequality, i.e.,
\[
  |h_k-\tilde h_k| \le r_0^{-1}\|\Proxrphi(x_k-r_0 g_k) - \tilde{x}_k(u,r_0)\|,
\]
where again $u\in\Proxrphi(x_k-r_0 g_k,\delta)$ for some $\delta > 0$.  
As we now demonstrate, Algorithm~\ref{alg:inex.M.prox} yields an element that approximates the true proximity operator, $\Proxrphi(x,\delta)$, to arbitrary precision.   For this result, we recall the definition $\varphi(\cdot)\coloneqq \tfrac{1}{2r}\|\cdot-x\|^2+\phi(\cdot)$ from the proof of Theorem~\ref{thm:delta.Prox.Thm}. 

\begin{lemma}\label{app:lemma.tech}
Algorithm~\ref{alg:inex.M.prox} converges in finitely many iterations.  Moreover, if $\alpha_1 \le \sqrt{2}$ and Algorithm~\ref{alg:inex.M.prox} exits at iteration $\ell$, i.e., 
\begin{equation}\label{eq:lemma.tech.1}
  \|u_\ell - u_{\ell+1}\|_a \le \varepsilon,
\end{equation}
then the following error bound holds
\[
  \|u_\ell-\Proxrphi(x)\|\le \alpha_1^{-1/2}\left(1-\tfrac{1}{2}\alpha_1^{-2}\right)^{-1}\varepsilon.
\]
\end{lemma}
\begin{proof}
By \cite[Corollary~28.9]{bauschke2017monotone}, $u_\ell$ converges strongly and therefore there exists $\ell$ for which \eqref{eq:lemma.tech.1} holds for given $\epsilon>0$.
Denote the proximal gradient operator associated with the $a$-inner product by
\[
  G_a(y,t) = \tfrac{1}{t}\left(y-\Prox_{t\phi}^a\left(y-\tfrac{t}{r}A^{-1}(y-x)\right)\right)
\]
and note that $y\mapsto G_a(y,t)$ is strongly monotone for all $t\in(0,2r/\alpha_1^2)$
\cite[Lemma~2]{baraldi2024local}.
Consequently, $y\mapsto G_a(y,r)$ is strongly monotone since $\alpha_1\le\sqrt{2}$ and
\[
  a(G_a(u_\ell,r)-G_a(p,r),u_\ell-p) \ge \tfrac{1}{r}\left(1-\tfrac{1}{2}\alpha_1^{-2}\right)\|u_\ell-p\|_a,
\]
where $p=\Proxrphi(x)$.  The optimality of $p$ for $\varphi$ ensures that $G_a(p,r)=0$ and so
\[
  \tfrac{1}{r}\left(1-\tfrac{1}{2}\alpha_1^{-2}\right)\|u_\ell-p\|_a \le \|G_a(u_\ell,r)\|_a = \tfrac{1}{r}\|u_{\ell+1}-u_\ell\|_a.
\]
The result then follows from \eqref{eq:lemma.tech.1} and the equivalence of $\|\cdot\|_a$ and $\|\cdot\|$.
\end{proof}



In Theorem~\ref{thm:inexact.M.prox}, we demonstrate that Algorithm~\ref{alg:inex.M.prox} applied to $x_k-r g_k$ outputs $u_{\ell+1}$ satisfying the following inequality
\[
\begin{aligned}
  \inp{g_k}{u_{\ell+1}-x_k}&+\tfrac{1}{2r}\|u_{\ell+1} - x_k\|^2 + \phi(u_{\ell+1}) \\
  &\le \inp{g_k}{z-x_k}+\frac{1}{2r}\|z - x_k\|^2 + \phi(z) +\delta\|z-u_{\ell+1}\|
\end{aligned}
\]
for every $z\in X$, or equivalently
\[
u_{\ell+1}\in \Proxrphi(x_k-rg_k,\delta). 
\]
Before presenting the main theorem of this section, we prove the following preliminary result, which we use to facilitate the proof of our main result.

\begin{lemma}\label{lem:M.weight.1}
Consider the sequence $\{u_\ell\}$ generated by Algorithm~\ref{alg:inex.M.prox}. Then,
\begin{equation}\label{eq:M.weight.1}
\frac{1}{r}(A-I)(u_\ell-u_{\ell+1}) \in \partial\varphi(u_{\ell+1}),
\end{equation}
which by definition is equivalent to
\begin{align}
\varphi(z) 
&\ge \varphi(u_{\ell+1}) + \frac{1}{r}\inp{(A-I)(u_\ell-u_{\ell+1})}{z-u_{\ell+1}}
\quad\forall\,z\in X.
\label{eq:diff}
\end{align}
\end{lemma}
\begin{proof}
Recall that $u_{\ell+1} = \Proxrphi^a(u_\ell-A^{-1}(u_\ell-x))$.
By Theorem~\ref{thm:delta.Prox.Thm}, we have that
\begin{align*}
\frac{1}{r}(u_\ell-u_{\ell+1}-A^{-1}(u_\ell-x))& \in \partial_A \phi(u_{\ell+1}) = A^{-1}\partial\phi(u_{\ell+1})
\end{align*}
and by adding $(1/r)A^{-1}(u_{\ell+1}-x)$ on both sides of the previous set inclusion (in the Minkowski sense), we obtain
\begin{align*}
\frac{1}{r}(u_\ell-u_{\ell+1}-A^{-1}(u_{\ell}-u_{\ell+1}))& \in \frac{1}{r}A^{-1}(u_{\ell+1}-x)+ A^{-1}\partial\phi(u_{\ell+1})\\
&=A^{-1}\partial\varphi(u_{\ell+1}).
\end{align*}
This is equivalent to \eqref{eq:M.weight.1}, concluding the proof. 
\end{proof}

Finally, we present our main result, which shows that if $u_{\ell+1}$ satisfies the stopping conditions of Algorithm~\ref{alg:inex.M.prox}, then it is a $\delta$-proximity operator. 
\begin{theorem}\label{thm:inexact.M.prox}
If $u_{\ell+1}$ satisfies the stopping condition \eqref{eq:lemma.tech.1} with
\[
  \varepsilon \le r\delta\sqrt{\alpha_1} / (1+\alpha_2),
\]
then $u_{\ell+1}\in \Proxrphi(x,\delta)$, i.e., for all $z\in X$,
\begin{align*}
\varphi(z) + \delta \|z-u_{\ell+1}\|
&\ge \varphi(u_{\ell+1}).
\end{align*}
\end{theorem}

\begin{proof}
First note that the existence of $\alpha_1$ and $\alpha_2$ ensure that
\[
  \frac{1}{r}\|(A-I)(u_\ell-u_{\ell+1})\| \le \frac{1+\alpha_2}{r} \|u_\ell-u_{\ell+1}\| \le \frac{1+\alpha_2}{r\sqrt{\alpha_1}}\|u_\ell-u_{\ell+1}\|_a \le \frac{1+\alpha_2}{r\sqrt{\alpha_1}}\varepsilon \le \delta.
\]
Now, for any $z\in X$, we have that
\begin{align*}
\varphi(z) + \delta \|z-u_{\ell+1}\|&\ge \varphi(z) + \frac{1}{r}\|(A-I)(u_\ell-u_{\ell+1})\| \|z-u_{\ell+1}\| \\
&\ge \varphi(z) + \frac{1}{r}\inp{(A-I)(-u_\ell+u_{\ell+1})}{z-u_{\ell+1}}\\
&\ge \varphi(u_{\ell+1})
\end{align*}
where in the first line we applied the bound on $\delta$, in the second line we used the Cauchy-Schwarz inequality and for the final inequality we applied Lemma~\ref{lem:M.weight.1}.
\end{proof}

\section{Numerics}
\label{sec:numerics}


We apply our inexact weighted proximity operator within Algorithm~\ref{alg:cap} to solve optimal control of Burgers' equation:
\begin{equation}
    \min_{z\in L^2(\Omega)} \int_{\Omega}([S(z)]-w)^2(x)\dx + \frac{\alpha}{2}\int_{\Omega}z^2(x)\dx + \beta\int_{\Omega}|z|(x)\dx
\end{equation}
where $\Omega=(0,1)$ is the physical domain, $\alpha=10^{-4}$ and $\beta=10^{-2}$ are penalty parameters, $w(x)=-x^{2}$ is the target state, and $S(z)=u\in H^{1}(\Omega)$ solves the weak form of Burgers' equation
\begin{eqnarray*}
    -\nu u'' + uu' = z+f \quad \text{in } \Omega, \\
    u(0)=0, \quad u(1)=-1,
\end{eqnarray*}
where $f=2(\nu+x^3)$ and $\nu =0.08$. We discretize the state $u$ and $z$ using continuous piecewise linear finite elements on a uniform mesh with $n=512$ intervals.  To compute $S(z)$, we solve the discretized Burgers' equation using Newton's method globalized with a backtracking line search.  We exit the Newton iteration when the relative residual falls below $10^{-4}\sqrt{\epsilon_{\rm mach}}$, where $\epsilon_{\rm mach}$ is machine epsilon.  We will refer to PDE solves using this tolerance as ``exact'' PDE solves.

Given that the controls are in $L^2(\Omega)$, we employ the weighted dot product $\inp{\cdot}{\cdot}=\inp{\cdot}{\cdot}_M$, where $M\in\R^{n\times n}$ is the mass matrix and for Algorithm~\ref{alg:inex.M.prox}, we employ the dot product $a(\cdot,\cdot)=\inp{\cdot}{\cdot}_D$ weighted by the lumped mass matrix $D\in\R^{n\times n}$, i.e.,
\[
M=\frac{h}{6}\begin{pmatrix}
4 & 1 & \ldots & 0 & 0\\
1 & 4 & \ldots & 0 & 0\\
\vdots & \vdots & \ddots & \vdots &\vdots\\
0 & 0 & \ldots & 4 & 1\\
0 & 0 & \ldots & 1 & 4\\
\end{pmatrix} \in \mathbb{R}^{n\times n}\quad \text{and}\quad D=h\begin{pmatrix}
\frac{5}{6} & 0 & \ldots & 0 & 0\\
0 & 1 & \ldots & 0 & 0 \\
\vdots & \vdots & \ddots & \vdots & \vdots\\
0 & 0 & \ldots & 1 & 0 \\
0 & 0 & \ldots & 0 & \frac{5}{6} \\
\end{pmatrix}\in \mathbb{R}^{n\times n}.
\]
Recall that the $A$ operator generated by the bilinear form $a(\cdot,\cdot)$ is $M^{-1}D$. 
 For this $A$, $\alpha_1=1$ and $\alpha_2=3$.  Consequently, the preceding theory applies.
Furthermore, we approximate the $L^1$-norm by the quantity
\[
  \beta\int_\Omega |z|(x)\dx \approx \phi(z)=\beta h(\tfrac{5}{6} |z_1| + |z_2| + \ldots + |z_{n-1}| + \tfrac{5}{6}|z_n|).
\]
We notice that the $D$-weighted proximity operator of $\phi$ is the usual soft-thresholding operator
\[
  \Proxrphi^D(z) = \textup{sign}(z)\odot\max\{|z|-\beta r,0\}.
\]
\if(0)
\subsection{$L_1$-norm}
Assume that is the $L_1$-norm, i.e.,
\begin{align*}
    \phi(y):\mathbb{R}^n &\to \mathbb{R} \\
            y &\mapsto \|y\|_1\coloneqq \sum_{i=1}^{n}|y_i|
\end{align*}
where $y=(y_1,\ldots, y_n)$. Thus,
\begin{align*}
\Prox^D_{r \phi}(x) &= \Argmin_{y} \left\{\frac{1}{2r}\|y - x\|_D^2 + \phi(y)\right\} \\
&= \Argmin_{y} \left\{\frac{1}{2r}\sum_{i=1}^nd_{ii}(y_i-x_i)^2 + \sum_{i=1}^{n}|y_i|\right\} \\
&= \Argmin_{y} \left\{\sum_{i=1}^n\left[\frac{d_{ii}}{2r} (y_i-x_i)^2 + |y_i|\right]\right\} \\
& = (\ldots, \Argmin_{y_i\in \mathbb{R}} \left\{\frac{d_{ii}}{2r} (y_i-x_i)^2 + |y_i|\right\} ,\ldots )
\end{align*}

Now notice that
\begin{align*}
\Argmin \left\{ \frac{d_{ii}}{2r} (y_i-x_i)^2 + |y_i| \right\} &= \Argmin\left\{  \frac{1}{2} (y_i-x_i)^2 + \frac{r}{d_{ii}}|y_i| \right\} \\
&= \text{sign } (x_i)\max \left\{|x_i|-\frac{r}{d_{ii}} ,0\right\}
\end{align*}
\fi

In Table~\ref{tbl:1}, we summarize the performance of Algorithm~\ref{alg:cap} for different values of scaling ($\kgrad$) for the inexact gradient tolerance.
To approximately solve the trust-region subproblem \eqref{eq:tr-sub}, we employ the nonlinear conjugate gradient solver introduced in \cite[Algorithm~4]{baraldi2024efficient}.
The table includes the wallclock time in seconds ({\tt time (s)}), the number of trust-region iterations ({\tt iter}), the number of evaluations of $f$ ({\tt obj}), the number of evaluations of $\nabla f$ ({\tt grad}), the number of applications of the Hessian $\nabla^2 f$ ({\tt hess}), the number of evaluations of the proximity operator ({\tt prox}), and the average number of iterations of Algorithm~\ref{alg:inex.M.prox} ({\tt av-piter}).  As $\kgrad$ decreases, we require additional accuracy from our approximate proximity operator as computed by Algorithm~\ref{alg:inex.M.prox}.  As seen in Table~\ref{tbl:1}, our implementation of Algorithm~\ref{alg:cap} is robust to inexact proximity operators.
\begin{table}[!ht]
\centering
\begin{tabular}{l r r r r r r r}
$\kgrad$ & {\tt time (s)} & {\tt iter} & {\tt obj} & {\tt grad} & {\tt hess} & {\tt prox} & {\tt av-piter} \\
   \hline
 {\tt  1e2} & {\tt 0.4708} & {\tt 18} & {\tt 37} & {\tt 19} & {\tt 173} & {\tt 347} & {\tt  10.78} \\ 
 {\tt  1e1} & {\tt 0.4447} & {\tt 16} & {\tt 33} & {\tt 17} & {\tt 141} & {\tt 281} & {\tt  16.63} \\ 
 {\tt  1e0} & {\tt 0.5229} & {\tt 18} & {\tt 37} & {\tt 19} & {\tt 173} & {\tt 347} & {\tt  28.50} \\ 
 {\tt 1e-1} & {\tt 0.4247} & {\tt 15} & {\tt 31} & {\tt 16} & {\tt 125} & {\tt 248} & {\tt  47.07} \\ 
 {\tt 1e-2} & {\tt 0.5815} & {\tt 17} & {\tt 35} & {\tt 19} & {\tt 157} & {\tt 315} & {\tt  60.00} \\ 
 {\tt 1e-3} & {\tt 0.3498} & {\tt 13} & {\tt 27} & {\tt 15} & {\tt  93} & {\tt 183} & {\tt  53.15} \\ 
 {\tt 1e-4} & {\tt 0.4034} & {\tt 13} & {\tt 27} & {\tt 17} & {\tt  93} & {\tt 185} & {\tt  69.31} \\
 \hline
\end{tabular}
\caption{Results for the Burgers' equation for $\kgrad\in\{100,10,0.1,0.01,0.001,0.0001\}$}
\label{tbl:1}
\end{table}

To conclude, we incorporate inexact PDE solves with our inexact proximity operator computations.  In particular, we terminate the Newton iteration for solving Burgers' equation when the relative residual falls below
\[
  \min\{10^{-2},\tau\},
\]
where $\tau$ is computed by the trust-region algorithm and corresponds to the right-hand sides in Assumptions~\ref{assump:inexact.object} and \ref{assump:inexact.gradient}.  We do this fixing to $\kgrad=1$ and choosing $\kobj=10^3$.  Since the dominant cost of Newton's method is the linear system solves at each iteration, we compare the average number of linear system solves per iteration between our runs with exact and inexact PDE solves.  In particular, our method averaged 5.3125 linear system solves per trust-region iteration when using inexact PDE solves, compared with 7.7222 linear system solves when using exact PDE solves.  The reduction in linear system solves is a byproduct of our relaxed tolerances.  In particular, many of the PDE solves{---the primary cost in PDE-constrained optimization---}only require the relative residual to be smaller than $10^{-2}$.  {Relaxing the PDE solver tolerance has the potential to make previously intractable problems solvable.}

\section{Conclusion and Future Work}

In this work, we introduced an inexact proximity operator---motivated by the $\delta$-Fr\'{e}chet subdifferential---for use within the inexact trust-region algorithm from \cite{Baraldi2023}.  Additionally, we extended the inexact trust-region algorithm, Algorithm~\ref{alg:cap}, from \cite{Baraldi2023}, to leverage our inexact proximity operator. Our numerical results suggest that our algorithm is robust to inexact proximity operator evaluations.
As future research, we hope to develop similar theory that extends beyond the weighted proximity operator problem studied here.

\bibliographystyle{siam}
\bibliography{LeandroFariasMaia.bib}

\end{document}